\documentclass{elsarticle}%
\usepackage{amsfonts}
\usepackage{amsthm}
\usepackage{amsmath}
\usepackage{amssymb}
\usepackage{bbold}
\usepackage{mathtools}
\usepackage{xcolor}
\usepackage{soul}
\usepackage{mathrsfs}
\usepackage{tabularx}
\usepackage[top=1in, bottom=1in, left=1in, right=1in]{geometry}
\usepackage{graphicx}%
\providecommand{\U}[1]{\protect\rule{.1in}{.1in}}

\newtheorem{theorem}{Theorem}[section]
\newtheorem{lemma}{Lemma}[section]
\newtheorem{corollary}{Corollary}[section]

\theoremstyle{definition}

\theoremstyle{remark}

\newtheorem{example}{Example}[section]
\DeclareMathOperator*{\esssup}{ess\,sup}
\numberwithin{equation}{section}
\begin{document}
	\begin{frontmatter}
		
		\title{Delayed  Gronwall inequality with weakly the singular kernel}
		
		

		\author[]{Javad A. Asadzade}
		\ead{javad.asadzade@emu.edu.tr}
		\author[]{Jasarat J. Gasimov}
		\ead{jasarat.gasimov@emu.edu.tr}
		\author[]{Nazim I. Mahmudov}
		\ead{nazim.mahmudov@emu.edu.tr}
		\cortext[cor1]{Corresponding author}

		\address{Department of Mathematics, Eastern Mediterranean University, Mersin 10, 99628, T.R. North Cyprus, Turkey}
		\address{	Research Center of Econophysics, Azerbaijan State University of Economics (UNEC), Istiqlaliyyat Str. 6, Baku 1001, Azerbaijan}
		

		\begin{abstract}
			
		Delay Gronwall inequality with a weakly singular kernel has been a subject of interest in various mathematical studies. In this article, we will delve into the consideration of this inequality and its application in the study continuity of the state trajectory for a  Volterra integral equation with delay. Using delay Gronwall inequality with a singular kernel, we investigate the behavior and properties of the state trajectory if there are delays. This analysis aims to improve our understanding of the dynamics associated in del Volterra integral equations with delay. In addition, we present a comprehensive example demonstrating the practical significance of our results.
			\noindent
		\end{abstract}
		\begin{keyword}
		Gronwall inequality, well-posedness, singular kernel
		\end{keyword}
	\end{frontmatter}

	\section{INTRODUCTION}
	\label{Sec:intro}
	The delay Gronwall inequality with a singular kernel is a fundamental result in the field of mathematical analysis and is widely used in various areas of mathematics, physics, and engineering. The inequality provides a powerful tool for estimating the growth and behavior of functions and has applications in diverse fields such as differential equations,delay systems, mathematical modeling, control theory, and stochastic processes.
	
	The Gronwall inequality, named after American mathematician Thomas H. Gronwall, establishes an upper bound for the growth of a function based on its integral and the behavior of a singular kernel. The inequality has proven to be an indispensable tool for studying the properties of solutions to differential equations and integral equations.
	
	The delay Gronwall inequality with a singular kernel offers a rigorous mathematical framework to analyze the growth and stability of functions subject to delays and singular influences. By considering the integral of a function multiplied by the singular kernel, the inequality provides an upper bound on the growth rate of delayed functions, enabling researchers to derive essential bounds and estimates. This inequality is indispensable for studying the stability and convergence properties of delayed differential equations, integral equations, and time-delay systems.

	Furthermore, the delay Gronwall inequality with a singular kernel has profound implications in various scientific fields. In control theory, it plays a vital role in analyzing the stability and performance of control systems subjected to delays, ensuring robustness and preventing undesired behaviors. In signal processing and communication systems, it assists in understanding the impact of time delays on signal transmission and reception, contributing to the design of efficient and reliable communication networks. Its versatility and broad applicability make the delay Gronwall inequality with a singular kernel a valuable asset in both theoretical and applied mathematics.
	
	In this paper, we aim to provide a comprehensive overview of the delay Gronwall inequality with a singular kernel. We will explore its mathematical formulation, discuss its properties, and examine its applications in different areas of mathematics and science. By understanding the intricacies of this inequality, researchers and practitioners can employ it effectively to tackle various problems and gain insights into the behavior of complex systems.
	
	In recent research \cite{1}-\cite{15}, different generalizations of the Gronwall inequalities have been developed. In particular, Ping Lin and Jiongmin Yong \cite{1} derived a notable inequality under certain conditions. 
	
	We will denote the following necessary theorem, which we use it for proof of the delayed Gronwall's inequality for singular integral.
	\begin{lemma}(\cite{1})\label{lemma1}
		Let $\beta \in (0, 1)$ and $q > 1/\beta$. Let $L(\cdot)$, $a(\cdot)$, $y(\cdot)$ be nonnegative functions with $L(\cdot) \in L^q(0, T)$ and $a(\cdot)$, $y(\cdot) \in L^{\frac{q}{q-1}}(0, T)$. Suppose 
		\begin{align}\label{as}
			y(t) \leq a(t) + \int_{0}^{t} \frac{L(s)y(s)}{(t - s)^{1-\beta}} ds, \quad a.e. \quad t \in [0, T].
		\end{align}
		Then there exists a constant $K>0$ such that
		\begin{align*}
			y(t) \leq a(t) + K\int_{0}^{t} \frac{L(s)a(s)}{(t - s)^{1-\beta}} ds, \quad a.e. \quad t \in [0, T].
		\end{align*}
		
	\end{lemma}
	
	Another interesting scenario, if we replace  assumption \eqref{as}  with the following condition:
	
	\begin{align}\label{as1}
		\xi(t)\leq \vartheta(t)+\int_{0}^{t}\frac{L(s)\xi(s)}{(t-s)^{1-\nu}}ds+\int_{0}^{t}\frac{L(s)\xi(s-h)}{(t-s)^{1-\nu}}ds, \quad \text{a.e.,} \quad t\in [0,T].
	\end{align}
	
	Then, there exists a constant $K>0$ such that
	
	\begin{align}
		\xi(t)\leq&\vartheta_{n}(t)+K\int_{0}^{t}\frac{L(S)\vartheta(s)}{(t-s)^{1-\nu}}ds, \quad t\in [0,T].
	\end{align}
	Where 
	\begin{align*}
		\vartheta_{n}(t)=&\vartheta(t)+K\sum_{k=1}^{n}\int_{0}^{t-kh}\frac{L(s)\vartheta(s)}{(t-kh-s)^{1-\nu}}ds+K\sum_{k=0}^{n-1}\int_{h}^{t-kh}\frac{L(s)\vartheta(s-h)}{(t-kh-s)^{1-\nu}}ds.
	\end{align*}
	The assumption \eqref{as1} is considered natural and plays a crucial role in investigating the well-posedness problem discussed in Section 4. These results highlight the significant advancements in understanding the properties and generalizations of Gronwall inequalities, paving the way for further exploration and applications in various mathematical analyses.
	\section{Preliminaries} In the forthcoming section, we will provide initial findings that will prove beneficial in subsequent analysis. To begin, consider a predetermined time horizon denoted as $T>0$. We will now define the subsequent spaces:
	\begin{align*}
		L^{p}(0,T;R^{n})=\bigg\lbrace \phi :[0,T]\to R^{n} \quad | \quad \phi(\cdot) \quad is \quad measurable, \\
		\Vert \phi(\cdot)\Vert_{p}\equiv \bigg(\int_{0}^{T} \vert \phi(t)\vert^{p}dt\bigg)^{1/p}<\infty\bigg\rbrace, 1\leq p<\infty,
	\end{align*}
	\begin{align*}
		L^{\infty}(0,T;R^{n})=\bigg\lbrace \phi :[0,T]\to R^{n} \quad | \quad \phi(\cdot) \quad is \quad measurable, \quad \Vert \phi(\cdot)\Vert_{\infty}\equiv \esssup_{t\in[0,T]}\vert \phi(t)\vert<\infty\bigg\rbrace.
	\end{align*}
	Also, we define
	\begin{align*}
		L^{p+}(0,T;R^{n})=\bigcup_{r>p} L^{r}(0,T;R^{n}), \quad  1\leq p<\infty,\\
		L^{p-}(0,T;R^{n})=\bigcap_{r<p} L^{r}(0,T;R^{n}), \quad  1<p<\infty.
	\end{align*}
	In the subsequent analysis, we utilize the notation $\Delta=\lbrace (t,s)\in [0,T]^{2} | 0\leq s<t\leq T\rbrace$. It is important to note that the "diagonal line" represented by $\lbrace (t,t)| t\in [0,T]\rbrace$ does not belong to $\Delta$. Consequently, if we consider a continuous mapping $\phi :\Delta\to R^{n}$ where $(t,s)\mapsto\phi(t,s)$, the function $\phi(\cdot,\cdot)$ may become unbounded as the difference $\vert t-s\vert \to 0$.
	
	In the article, we adopt the notation $t_{1} \vee t_{2}=\max \lbrace t_{1},t_{2}\rbrace$ and $t_{1} \wedge t_{2}=\min \lbrace t_{1},t_{2}\rbrace$, for any $t_{1},t_{2}\in R$. Notably, $t^{+}=t\vee 0$.

	\begin{lemma}\label{l1}(Lemma 2.1, page 138, \cite{1})
		Let $p,q,r\geq 1$ satisfy $\frac{1}{p}+1=\frac{1}{q}+\frac{1}{r}$. Then for any $ f(\cdot)\in L^{q}(R^{n}),g(\cdot)\in L^{r}(R^{n})$,
		\begin{align}
			\Vert f(\cdot)*g(\cdot)\Vert_{L^{p}(R^{n})}\leq \Vert f(\cdot)\Vert_{L^{q}(R^{n})}\Vert g(\cdot)\Vert_{L^{r}(R^{n})}.
		\end{align}
	\end{lemma}
	\begin{corollary}\label{c1}(Corollary 2.2, page 138, \cite{1})
		Let $\beta\in (0,1)$, $1\leq r<\frac{1}{1-\beta}$, and $\frac{1}{p}+1=\frac{1}{q}+\frac{1}{r}, \quad p,q\geq1$. Then for any $a<b, \quad 0<\delta\leq b-a$, and $\varphi(\cdot)\in L^{q}(a,b)$,
		\begin{align}
			\bigg(\int_{a}^{a+\delta}\bigg\vert \int_{a}^{t}\frac{\varphi(s)ds}{(t-s)^{1-\beta}}\bigg\vert^{p}dt\bigg)^{1/p}\leq \bigg(\frac{\delta^{1-r(1-\beta)}}{1-r(1-\beta)}\bigg)^{1/r}\Vert \varphi(\cdot)\Vert_{L^{q}(a,b)}.
		\end{align}
	\end{corollary}
	\section{Main result}  In this section, we will present the following theorem that demonstrates the delay Gronwall inequality with a singular kernel.
	\begin{theorem}\label{t12}
		Let $\nu\in(0,1)$, $h>0 $, and $q>\frac{1}{\nu}$. Suppose that $L(\cdot), \vartheta(\cdot),\xi(\cdot)$ be nonnegative functions with $\vartheta(t)=0$,  $L(t)=0 \quad t<0$, where $L(\cdot)\in L^{q}(0,T)$ and $\vartheta(\cdot),\xi(\cdot)\in L^{\frac{q}{q-1}}[-h,T)$. Assume
		\begin{equation}\label{eq1}
			\begin{cases}
				\xi(t)\leq \vartheta(t)+\int_{0}^{t}\frac{L(s)\xi(s)}{(t-s)^{1-\nu}}ds+\int_{0}^{t}\frac{L(s)\xi(s-h)}{(t-s)^{1-\nu}}ds, & a.e. \quad t\in[0,T],\\
				\xi(t)=0, & t<0.
			\end{cases}
		\end{equation}
		Then there is a constant $K>0$ such that
		\begin{align}
			\xi(t)\leq&\vartheta_{n}(t)+K\int_{0}^{t}\frac{L(s)\vartheta(s)}{(t-s)^{1-\nu}}ds,\quad a.e. \quad t\in [0,T],
		\end{align}
		where
		\begin{align}
			\vartheta_{n}(t)=&\vartheta(t)+K\sum_{k=1}^{n}\int_{0}^{t-kh}\frac{L(s)\vartheta(s)}{(t-kh-s)^{1-\nu}}ds+K\sum_{k=0}^{n-1}\int_{h}^{t-kh}\frac{L(s)\vartheta(s-h)}{(t-kh-s)^{1-\nu}}ds.
		\end{align}
	\end{theorem}
	\begin{proof}
		Let's contemplate the partition of $[0,T]$ by intervals of length $h$, where there exists an integer $n$ such that $0 < h < 2h < \dots < nh \leq T < (n+1)h$.
		
		$\bullet$ \textbf{Step 1.} Consider \eqref{eq1}  on the $0<t\leq h$,
		\begin{align*}
			\xi(t)\leq \vartheta(t)+\int_{0}^{t}\frac{L(s)\xi(s)}{(t-s)^{1-\nu}}ds+\int_{0}^{t}\frac{L(s)\xi(s-h)}{(t-s)^{1-\nu}}ds.
		\end{align*}
		Since $\xi(t)=0, \quad -h\leq t\leq 0$, it follows 
		\begin{align*}
			\xi(t)\leq \vartheta(t)+\int_{0}^{t}\frac{L(s)\xi(s)}{(t-s)^{1-\nu}}ds.
		\end{align*}
		By using Lemma \ref{lemma1}, we obtain 
		\begin{align}\label{1c}
			\xi(t)\leq \vartheta_{0}(t)+K_{0}\int_{0}^{t}\frac{L(s)\vartheta(s)}{(t-s)^{1-\nu}}ds,\quad t\in [0,h],
		\end{align}
		where $\vartheta_{0}(t)=\vartheta(t)$.
		
		$\bullet$ \textbf{Step 2.} Now, employing the same principle, we take  $h< t\leq 2h$,
		\begin{align*}
			\xi(t)\leq \vartheta(t)+\int_{0}^{t}\frac{L(s)\xi(s)}{(t-s)^{1-\nu}}ds+\int_{h}^{t}\frac{L(s)\xi(s-h)}{(t-s)^{1-\nu}}ds
		\end{align*}
		Substituting (\ref{1c}) into the above 
		
		\begin{align*}
			\xi(t)\leq& \vartheta(t)+\int_{0}^{t}\frac{L(s)\xi(s)}{(t-s)^{1-\nu}}ds+\int_{h}^{t}\frac{L(s)}{(t-s)^{1-\nu}}\bigg[\vartheta(s-h)+K_{0}\int_{0}^{s-h}\frac{L(\tau)\vartheta(\tau)}{(s-h-\tau)^{1-\nu}}d\tau \bigg]ds.
		\end{align*}
		By applying the Fubini's theorem, 
		\begin{align*}
			&\int_{h}^{t}\frac{L(s)}{(t-s)^{1-\nu}}\int_{0}^{s-h}\frac{L(\tau)\vartheta(\tau)}{(s-h-\tau)^{1-\nu}}d\tau ds=\int_{0}^{t-h}L(\tau)\bigg[ \int_{\tau+h}^{t}\frac{L(s)ds}{(t-s)^{1-\nu}(s-h-\tau)^{1-\nu}}\bigg]\vartheta(\tau)d\tau
		\end{align*}
		Utilizing the Holder inequality, and letting $r=\frac{s-h-\tau}{t-h-\tau}$ implies  
		\begin{align*}
			\int_{\tau+h}^{t}\frac{L(s)ds}{(t-s)^{1-\nu}(s-h-\tau)^{1-\nu}}\leq& \bigg(\int_{\tau+h}^{t}L(s)^{q}ds\bigg)^{\frac{1}{q}}\bigg(\int_{\tau+h}^{t}\frac{ds}{(t-s)^{(1-\nu)\frac{q}{q-1}}(s-h-\tau)^{(1-\nu)\frac{q}{q-1}}}\bigg)^{\frac{q-1}{q}}\\
			\leq&\Vert L(\cdot)\Vert_{q} \frac{1}{(t-h-\tau)^{2(1-\nu)-\frac{q-1}{q}}}\bigg(\int_{0}^{1}\frac{dr}{(1-r)^{(1-\nu)\frac{q}{q-1}} r^{(1-\nu)\frac{q}{q-1}}}\bigg)^{\frac{q-1}{q}}.
		\end{align*}
		Since $q> \frac{1}{\nu}$  that is equivalent to
		\begin{align}
			0<1-\frac{(1-\nu)q}{q-1}=\frac{\nu q-1}{q-1},
		\end{align}
		we get
		\begin{align*}
			&\int_{\tau+h}^{t}\frac{L(s)ds}{(t-s)^{1-\nu}(s-h-\tau)^{1-\nu}}\leq \frac{\Vert L(\cdot)\Vert_{q}}{(t-h-\tau)^{2(1-\nu)-\frac{q-1}{q}}} B\bigg(\frac{\nu q-1}{q-1}, \frac{\nu q-1}{q-1}\bigg)^{\frac{q-1}{q}} \equiv \frac{K_{1}}{(t-h-\tau)^{1-\nu_{1}}},
		\end{align*}
		where, $B(\cdot,\cdot)$ is the well-known Beta function.
		
		Therefore,
		\begin{align}\label{sey}
			\xi(t)\leq&\vartheta(t)+\int_{0}^{t}\frac{L(s)\xi(s)}{(t-s)^{1-\nu}}ds+\int_{h}^{t}\frac{L(s)\vartheta(s-h)}{(t-s)^{1-\nu}}ds+K_{1}\int_{0}^{t-h}\frac{L(s)\vartheta(s)}{(t-h-s)^{1-\nu_{1}}}ds,
		\end{align}
		such as
		\begin{align*}
			&K_{1}=\Vert L(\cdot)\Vert_{q}  B\bigg(\frac{\nu q-1}{q-1}, \frac{\nu q-1}{q-1}\bigg)^{\frac{q-1}{q}},\\
			&\nu_{1}=1-\Big(2(1-\nu)-\frac{q-1}{q}\Big)=\nu+(\nu-\frac{1}{q})>\nu.
		\end{align*}
		By using the following inequality in \eqref{sey}
		\begin{align}\label{2c}
			\frac{1}{(t-s)^{1-\nu_{1}}}\leq \frac{C}{(t-s)^{1-\nu}},  \quad 0\leq s<t\leq T,
		\end{align}
		we receive
		\begin{align*}
			\xi(t)\leq&\vartheta(t)+\int_{0}^{t}\frac{L(s)\xi(s)}{(t-s)^{1-\nu}}ds+\int_{h}^{t}\frac{L(s)\vartheta(s-h)}{(t-s)^{1-\nu}}ds+K_{1}\int_{0}^{t-h}\frac{L(s)\vartheta(s)}{(t-h-s)^{1-\nu}}ds.
		\end{align*}
		By using Lemma \ref{lemma1}, we achieve
		
		\begin{align}\label{3c}
			\xi(t)\leq&\vartheta_{1}(t)+K_{1}\int_{0}^{t}\frac{L(s)\vartheta(s)}{(t-s)^{1-\nu}}ds,
		\end{align}
		where 
		\begin{align*}
			\vartheta_{1}(t)=\vartheta(t)+K_{1}\int_{h}^{t}\frac{L(s)\vartheta(s-h)}{(t-s)^{1-\nu}}ds+K_{2}\int_{0}^{t-h}\frac{L(s)\vartheta(s)}{(t-h-s)^{1-\nu}}ds.
		\end{align*}
		$\bullet$ \textbf{Step 3.} Subsequently, we examine the range where $2h\leq t\leq 3h$,
		\begin{align*}
			\xi(t)\leq \vartheta(t)+\int_{0}^{t}\frac{L(s)\xi(s)}{(t-s)^{1-\nu}}ds+\int_{h}^{t}\frac{L(s)\xi(s-h)}{(t-s)^{1-\nu}}ds
		\end{align*}
		Again substituting  \eqref{3c} into the last expression
		\begin{align*}
			\xi(t)\leq& \vartheta(t)+\int_{0}^{t}\frac{L(s)\xi(s)}{(t-s)^{1-\nu}}ds+\int_{h}^{t}\frac{L(s)}{(t-s)^{1-\nu}}\bigg[\vartheta_{1}(s-h)+K_{1}\int_{0}^{s-h}\frac{L(\tau)\vartheta(s)}{(s-h-\tau)^{1-\nu}}d\tau\bigg]ds\\
			&\text{By using the Fubini's theorem and inequality \eqref{2c}}\\
			=&\vartheta(t)+K_{3}\int_{0}^{t}\frac{L(s)\xi(s)}{(t-s)^{1-\nu}}ds+K_{3}\int_{h}^{t}\frac{L(s)\vartheta(s-h)}{(t-s)^{1-\nu}}ds
			+K_{3}\int_{0}^{t-h}\frac{L(s)\vartheta(s)}{(t-h-s)^{1-\nu}}ds\\
			+&K_{3}\int_{h}^{t-h}\frac{L(s)\vartheta(s-h)}{(t-h-s)^{1-\nu}}ds	+K_{3}\int_{0}^{t-2h}\frac{L(s)\vartheta(s)}{(t-2h-s)^{1-\nu}}ds\\
		\end{align*}
		Applying Lemma \ref{lemma1}, we get
		\begin{align}\label{4c}
			\xi(t)\leq&\vartheta_{3}(t)+K_{4}\int_{0}^{t}\frac{L(s)\vartheta(s)}{(t-s)^{1-\nu}}ds,
		\end{align}
		where 
		\begin{align*}
			\vartheta_{3}(t)=&\vartheta(t)+K_{4}\int_{0}^{t-h}\frac{L(s)\vartheta(s)}{(t-h-s)^{1-\nu}}ds+K_{4}\int_{0}^{t-2h}\frac{L(s)\vartheta(s)}{(t-2h-s)^{1-\nu}}ds\\
			+&K_{4}\int_{h}^{t}\frac{L(s)\vartheta(s-h)}{(t-s)^{1-\nu}}ds+K_{4}\int_{h}^{t-h}\frac{L(s)\vartheta(s-h)}{(t-h-s)^{1-\nu}}ds	.
		\end{align*}
		As $T$ is finite, there exists $n\in N$ such that $T<(n+1)h$. Continuing the given recursive relationship in the same manner for $ nh\leq t\leq (n+1)h$, we obtain the following inequality:
		\begin{align*}
			\xi(t)\leq&\vartheta_{n}(t)+K\int_{0}^{t}\frac{L(s)\vartheta(s)}{(t-s)^{1-\nu}}ds,\quad a.e. \quad t\in [0,T],
		\end{align*}
		where, $K$ is the maximum value of $K_{i}$ for $i\in N$, and
		\begin{align*}
			\vartheta_{n}(t)=\vartheta(t)+K\sum_{k=1}^{n}\int_{0}^{t-kh}\frac{L(s)\vartheta(s)}{(t-kh-s)^{1-\nu}}ds+K\sum_{k=0}^{n-1}\int_{h}^{t-kh}\frac{L(s)\vartheta(s-h)}{(t-kh-s)^{1-\nu}}ds.
		\end{align*}
	\end{proof}
	\section{Well-posedness in $L^{p}$ space and continuity of the state trajectory.}
	In this section, we consider  the continuity of the state trajectory for the delayed Volterra integral equation with initial condition, and we will show the application of the delayed Gronwall inequality.
	\begin{equation}\label{eq3}
		\begin{cases}
			\xi(t)= \zeta(t)+\int_{0}^{t} \frac{\kappa(t,s,\xi(s),\xi(s-h),\upsilon(s))}{(t-s)^{1-\nu}}ds, \quad a.e.\quad  t\in [0,T]\\
			\xi(t)=0,\quad -h\leq t\leq 0,\quad h\geq0.
		\end{cases}
	\end{equation}
	We present a set of the following conditions that apply to the generator function $\kappa(\cdot,\cdot,\cdot,\cdot,\cdot)$ used in our state equation. Let $U$ be a separable metric space with the metric $d$ ,which could be a nonempty bounded or unbounded set in $R^{n}$ with the metric induced by the usual Euclidean norm. With the Borel $\sigma$-field, $U$ is regarded as a measurable space. Let $u_{0} \in U$ be fixed. For any $p \geq 1$, we define
	\begin{align*}
		\mathcal{U}_{p}[0, T] = \bigl\{u : [0, T] \rightarrow U \bigm| u(\cdot ) \quad is \quad measurable,\quad  d (u(\cdot ), u_{0}) \in L^{p}(0, T)\}.
	\end{align*}
	(H): Let $\kappa: [0,T]\times [0,T]\times R^{n}\times R^{n}\times U\rightarrow R^{n}$  be  a transformation with $(t,s)\mapsto \kappa(t,s,\xi,\xi_{h},\upsilon)$ being measurable, $(\xi,\xi_{h})\mapsto \kappa(t,s,\xi,\xi_{h},\upsilon)$  being continuously differentiable,  $(\xi,\xi_{h},\upsilon)\mapsto \kappa(t,s,\xi,\xi_{h},\upsilon)$ , $(\xi,\xi_{h},\upsilon)\mapsto \kappa_{\xi}(t,s,\xi,\xi_{h},\upsilon)$ and  $(\xi,\xi_{h},\upsilon)\mapsto \kappa_{\xi_{h}}(t,s,\xi,\xi_{h},\upsilon)$ being continuous. There exist non-negative functions. $L_{0}(\cdot), L(\cdot)$ with
	\begin{align}\label{eq4}
		&L_{0}(\cdot)\in L^{\frac{1}{\nu}{+}}(0,T),\quad L(\cdot)\in L^{\frac{p}{p\nu-1}{+}}(0,T)
	\end{align}
	for some $p>\frac{1}{\nu}$ and $\nu \in (0,1), \upsilon_{0}\in U$.
	\begin{align}\label{eq5}
		\vert \kappa(t,s,0,0,\upsilon_{0})\vert \leq L_{0}(s), \quad t\in [0,T],
	\end{align}
	\begin{align}\label{eq5}
		\vert \kappa(t,s,\xi,\xi_{h},\upsilon)-\kappa(t,s,\xi^{\prime},\xi^{\prime}_{h},\upsilon^{\prime})\vert
		\leq L(s)[\vert \xi-\xi^{\prime}\vert+\vert \xi_{h}-\xi^{\prime}_{h}\vert+d(\upsilon,\upsilon^{\prime})],\\
		for \quad (t,s)\in \Delta, \quad \xi,\xi^{\prime}, \xi_{h},\xi^{\prime}_{h}\in R^{n}, \quad \upsilon,\upsilon^{\prime}\in U\nonumber.
	\end{align}
	We point out (4.3)-(4.4) declare
	\begin{align}\label{eq6}
		\vert  \kappa(t,s,\xi,\xi_{h},\upsilon)\vert \leq L_{0}(s)+L(s)[\vert \xi\vert +\vert \xi_{h}\vert +d(\upsilon,\upsilon_{0})],\quad (t,\xi,\xi_{h},\upsilon)\in [0,T]\times R^{n}\times R^{n}
		\times U.
	\end{align}
	\begin{align}\label{eq7}
		\vert \kappa(t,s,\xi,\xi_{h},\upsilon)-\kappa(t^{\prime},s,\xi,\xi_{h},\upsilon)\vert\leq K\omega (\vert t-t^{\prime}\vert)(1+\vert \xi\vert +\vert \xi_{h}\vert), \quad t,t^{\prime}\in[0,T],\quad \xi,\xi_{h}\in R^{n}, \quad \upsilon\in U,
	\end{align}
	for some modulus of continuity $\omega(\cdot)$.
	Moreover, it is evident that $L$ is included in a smaller space, compared to the space to which $L_{0}$ belongs.
	
	We will now demonstrate the  well-posedness  of the state equation (4.1) when considering  $L^{p}$ spaces.
	
	\begin{theorem}\label{t4}
		Assume that (H) satisfies with some $p\geq 1$ and $\nu\in(0,1)$. Hence for each $\zeta(\cdot)\in L^{p}(-h,T;R^{n})$ and $\upsilon\in \mathscr{U}^{p}[0,T]$, (4.1) admits a unique solution $\xi(\cdot)\equiv \xi(\cdot, \zeta(\cdot), \upsilon(\cdot))\in L^{p}(-h,T;R^{n})$,
		and  the following estimate satisfies:
		\begin{align}\label{eq8}
			\Vert \xi(\cdot)\Vert_{p}\leq \Vert \zeta(\cdot)\Vert_{p}+K\Big(1+\Vert d(\upsilon(\cdot), \upsilon_{0})\Vert_{L^{p}(0,T)}\Big).
		\end{align}
		If $(\zeta_{1}(\cdot),\upsilon_{1}(\cdot)), (\zeta_{2}(\cdot),\upsilon_{2}(\cdot))\in L^{p}(-h,T;R^{n})\times \mathscr{U}^{p}[0,T]$ and $\xi_{1}(\cdot), \xi_{2}(\cdot)$ are the solutions of (4.1) corresponding to $(\zeta_{1}(\cdot),,\upsilon_{1}(\cdot))$, and  $(\zeta_{2}(\cdot),\upsilon_{2}(\cdot))$, accordingly, then
		\begin{align}\label{eq9}
			&\Vert  \xi_{1}(\cdot)-\xi_{2}(\cdot)\Vert_{p}
			\leq K\Big\lbrace \Vert \zeta_{1}(t)-\zeta_{2}(t)\Vert_{p} \nonumber\\
			+&\bigg[\int_{0}^{T}\bigg(\int_{0}^{t}\frac{\vert \kappa(t,s,\xi_{1}(s),\xi_{1}(s-h),\upsilon_{1}(s))-\kappa(t,s,\xi_{2}(s),\xi_{2}(s-h),\upsilon_{2}(s))\vert}{(t-s)^{1-\nu}}ds\bigg)^{p}dt\bigg]^{1/p}\Big\rbrace.
		\end{align}
	\end{theorem}
	\begin{proof}
		Assuming a fixed function $\zeta(\cdot)$ in the space $L^{p}(-h,S; R^{n})$, and a control function $\upsilon$ belonging to the set $\mathscr{U}^{p}[0,T]$. For each $ z(\cdot)\in L^{p}(-h,S;R^{n})$ with $0<S\leq T$, such that $z(t)=0,$ $ -h\leq t\leq0.$ Define
		\begin{align*}
			\mathscr{T}[z(\cdot)](t)=\zeta(t)+\int_{0}^{t}\frac{\kappa(t,s,z(s),z(s-h),\upsilon(s))}{(t-s)^{1-\nu}}ds,\quad a.e.\quad t\in [0,S].
		\end{align*}
		Hence by Corollary \ref{c1}, for each $q\geq 1, 0\leq \varepsilon<\frac{\nu}{1-\nu}$, with $\frac{1}{p}+1=\frac{1}{q}+\frac{1}{1+\varepsilon}$,
		\begin{align}\label{eq10}
			&\Vert \mathscr{T}[z(\cdot)]\Vert_{L^{p}(-h,S;R^{n})}\leq \Vert \zeta(\cdot)\Vert_{L^{p}(-h,S;R^{n})}+\bigg[\int_{0}^{S}\bigg\vert\int_{0}^{t}\frac{\kappa(t,s,z(s),z(s-h),\upsilon(s))}{(t-s)^{1-\nu}}ds\bigg\vert^{p}dt\bigg]^{1/p}\nonumber\\
			\leq&\Vert \zeta(\cdot)\Vert_{p}+\bigg[\int_{0}^{S}\Big(\int_{0}^{t}\frac{\vert \kappa(t,s,z(s),z(s-h),\upsilon(s))\vert}{(t-s)^{1-\nu}}ds\Big)^{p}dt\bigg]^{1/p}\nonumber\\
			\leq&\Vert \zeta(\cdot)\Vert_{p}+\bigg[\int_{0}^{S}\Big(\int_{0}^{t}\frac{L_{0}(s)+L(s)[\vert z(s)\vert+\vert z(s-h)\vert+d(\upsilon(s),\upsilon_{0})]}{(t-s)^{1-\nu}}ds\Big)^{p}dt\bigg]^{1/p}\nonumber\\
			\leq&\Vert \zeta(\cdot)\Vert_{p}+\bigg(\frac{S^{1-(1+\varepsilon)(1-\nu)}}{1-(1+\varepsilon)(1-\nu)}\bigg)^{\frac{1}{1+\varepsilon}}\Vert L_{0}(\cdot)+L(\cdot)[\vert z(\cdot)\vert+\vert z(\cdot-h)\vert+d(\upsilon(\cdot),\upsilon_{0})]\Vert_{L^{q}(0,S)}.
		\end{align}
		We will examine three cases.
		
		\textbf{Case 1.} $p>\frac{1}{1-\nu}$. In this part, $\frac{1}{\nu}>\frac{p}{p-1}$ and $\frac{p}{1+\nu p}>1$. For every $\varepsilon\in (0,\frac{\nu}{1-\nu})$, that is equivalent to $(1-\nu)(1+\varepsilon)<1$, we obtain (without any constraint on $p>\frac{1}{1-\nu}$ for the current case)
		\begin{align*}
			\frac{1}{q}=\frac{1}{p}+1-\frac{1}{1+\varepsilon}<\frac{1}{p}+1-\frac{1}{1+\frac{\nu}{1-\nu}}=\frac{1}{p}+\nu<1, \quad \frac{1}{q}-\frac{1}{p}=1-\frac{1}{1+\varepsilon}<\nu.
		\end{align*}
		Therefore,
		\begin{align*}
			q\searrow\frac{p}{1+\nu p}>1, \quad \frac{p q}{p-q}\searrow\frac{1}{\nu},\quad as \quad \varepsilon\nearrow\frac{\nu}{1-\nu}.
		\end{align*}
		As $L_{0}(\cdot)$ belongs to $L^{\frac{p}{1+\nu p}+}(0,T)$ and $L(\cdot)$ belongs to $L^{\frac{1}{\nu}+}(0,T)$, we was able to find $\varepsilon$ close enough to $\frac{\nu}{1-\nu}$ such that $L_{0}(\cdot)$ belongs to $L^{q}(0,T)$ and $L(\cdot)$ belongs to $L^{\frac{p q}{p-q}}(0,T)$.Hence
		\begin{align}\label{eq11}
			&\Vert L_{0}(\cdot)+L(\cdot)[\vert z(\cdot)\vert+\vert z(\cdot-h)\vert
			+d(\upsilon(\cdot),\upsilon_{0})]\Vert_{L^{q}(0,S)}\leq \Vert L_{0}(\cdot)\Vert_{L^{q}(0,T)}\nonumber\\
			+&\Vert L(\cdot)\Vert_{L^{\frac{p q}{p-q}}(0,T)}\Vert \vert z(\cdot)\vert+\vert z(\cdot-h)\vert+d(\upsilon(\cdot),\upsilon_{0})\Vert_{L^{p}(0,S)}.
		\end{align}
		which yields
		\begin{align}\label{eq12}
			\Vert \mathscr{T}[z(\cdot)]\Vert_{L^{p}(-h,S;R^{n})}\leq &\Vert \zeta(\cdot)\Vert_{p}+\bigg(\frac{S^{1-(1+\varepsilon)(1-\nu)}}{1-(1+\varepsilon)(1-\nu)}\bigg)^{\frac{1}{1+\varepsilon}}\Big\lbrace \Vert L_{0}(\cdot)\Vert_{L^{q}(0,T)}\nonumber\\
			+&\Vert L(\cdot)\Vert_{L^{\frac{p q}{p-q}}(0,T)}\Vert \vert z(\cdot)\vert+\vert z(\cdot-h)\vert+d(\upsilon(\cdot),\upsilon_{0})\Vert_{L^{p}(0,S)}\Big\rbrace\\
			\leq &\Vert \zeta(\cdot)\Vert_{p}+\bigg(\frac{S^{1-(1+\varepsilon)(1-\nu)}}{1-(1+\varepsilon)(1-\nu)}\bigg)^{\frac{1}{1+\varepsilon}}\Big\lbrace \Vert L_{0}(\cdot)\Vert_{L^{q}(0,T)}\nonumber\\
			+&\Vert L(\cdot)\Vert_{L^{\frac{p q}{p-q}}(0,T)}(2\Vert z(\cdot)\Vert_{p}+\Vert d(\upsilon(\cdot),\upsilon_{0})\Vert_{L^{p}(0,T)})\Big\rbrace\nonumber
		\end{align}
		Hence, we can conclude that the operator $\mathscr{T}$ maps functions in $L^{p}(-h,S;R^{n})$ to functions in $L^{p}(-h,S;R^{n})$ for all $S$ belonging to the interval $(0,T]$.
		Moving forward, suppose  $\delta\in (0,T]$ be undetermined, and consider two functions $z_{1}(\cdot)$ and $z_{2}(\cdot)$ belonging to $L^{p}(-h,\delta;R^{n})$. We will investigate the following expression (using Corollary \ref{c1}):
		\begin{align*}
			&\Big\Vert \mathscr{T}[z_{1}(\cdot)](t)-\mathscr{T}[z_{2}(\cdot)](t)\Big\Vert_{L^{p}(-h,S;R^{n})}\equiv \Bigg[ \int_{0}^{\delta}\bigg\vert \mathscr{T}[z_{1}(\cdot)](t)-\mathscr{T}[z_{2}(\cdot)](t)\bigg\vert^{p}dt\Bigg]^{1/p}\\
			=&\Bigg[ \int_{0}^{\delta}\bigg\vert \int_{0}^{t}\frac{\kappa(t,s,z_{1}(s),z_{1}(s-h),\upsilon(s))-\kappa(t,s,z_{2}(s),z_{2}(s-h),\upsilon(s))}{(t-s)^{1-\nu}}ds\bigg\vert^{p}dt\Bigg]^{1/p}\\
			\leq&\Bigg[ \int_{0}^{\delta}\bigg( \int_{0}^{t}\frac{\vert \kappa(t,s,z_{1}(s),z_{1}(s-h),\upsilon(s))-\kappa(t,s,z_{2}(s),z_{2}(s-h),\upsilon(s))\vert}{(t-s)^{1-\nu}}ds\bigg)^{p}dt\Bigg]^{1/p}\\
			\leq& \bigg(\frac{\delta^{1-(1+\varepsilon)(1-\nu)}}{1-(1+\varepsilon)(1-\nu)}\bigg)^{\frac{1}{1+\varepsilon}}\Vert L(\cdot)\Vert_{L^{\frac{p q}{p-q}}(0,T)}\Big( \Vert z_{1}(\cdot)-z_{2}(\cdot)\Vert_{L^{p}(-h,\delta;R^{n})}\\
			+&\Vert z_{1}(\cdot-h)-z_{2}(\cdot-h)\Vert_{L^{p}(-h,\delta;R^{n})}\Big)\\
			\leq& 2\bigg(\frac{\delta^{1-(1+\varepsilon)(1-\nu)}}{1-(1+\varepsilon)(1-\nu)}\bigg)^{\frac{1}{1+\varepsilon}}\Vert L(\cdot)\Vert_{L^{\frac{p q}{p-q}}(0,T)} \Vert z_{1}(\cdot)-z_{2}(\cdot)\Vert_{L^{p}(-h,\delta;R^{n})}.
		\end{align*}
		Consider choosing a value of $\delta$ from the interval $(0,T]$ such that $2\bigg(\frac{\delta^{1-(1+\varepsilon)(1-\nu)}}{1-(1+\varepsilon)(1-\nu)}\bigg)^{\frac{1}{1+\varepsilon}}\Vert L(\cdot)\Vert_{L^{\frac{p q}{p-q}}(0,T)}<1$. This choice of $\delta$ ensures that the operator $\mathscr{T}: L^{p}(-h,\delta;R^{n})\to L^{p}(-h,\delta;R^{n})$ is a contraction, and therefore has a unique fixed point $\xi(\cdot)$ in $L^{p}(-h,\delta;R^{n})$. This fixed point is the unique solution of the state equation (4.1) on the interval $[0,\delta]$, with $\xi(t)=0$, $-h\leq t\leq 0$.
		
		Next, we focus on the state equation (4.1) over the interval $[0,2\delta]$. Let $z(\cdot)$ be any function belonging to $L^{p}(\delta,2\delta;R^{n})$, and define the following expression:
		\begin{align*}
			\mathscr{T}[z(\cdot)](t)=\zeta(t)+\int_{0}^{\delta}\frac{\kappa(t,s,\xi(s),\xi(s-h),\upsilon(s))}{(t-s)^{1-\nu}}ds+\int_{\delta}^{t}\frac{\kappa(t,s,z(s),z(s-h),\upsilon(s))}{(t-s)^{1-\nu}}ds, \quad a.e\quad t\in [\delta,2\delta].
		\end{align*}
		Denote $\bar{z}(\cdot)=\xi(\cdot)\chi_{[0,\delta)}(\cdot)+z(\cdot)\chi_{[\delta,2\delta]}(\cdot)\in L^{p}(-h,2\delta;R^{n}), \quad \bar{z}(\cdot-h)=\xi(\cdot-h)\chi_{[0,\delta)}(\cdot)+z(\cdot-h)\chi_{[\delta,2\delta]}(\cdot)\in L^{p}(-h,2\delta;R^{n})$. Hence similarly to (4.9) - (4.11)(with $S=2\delta$)
		\begin{align*}
			&	\bigg\Vert \zeta(\cdot)+\int_{0}^{\delta}\frac{\kappa(t,s,\xi(s),\xi(s-h),\upsilon(s))}{(\cdot-s)^{1-\nu}}ds\bigg\Vert_{L^{p}(\delta,2\delta;R^{n})}\\
			=&\Bigg[ \int_{\delta}^{2\delta}\bigg\vert \zeta(t)+\int_{0}^{\delta}\frac{\kappa(t,s,\xi(s),\xi(s-h),\upsilon(s))}{(t-s)^{1-\nu}}ds \bigg\vert^{p}dt\Bigg]^{1/p}\\
			\leq&\Big(\int_{\delta}^{2\delta}\vert \zeta(t)\vert^{p}dt\Big)^{1/p}+\Bigg[ \int_{\delta}^{2\delta}\bigg( \int_{0}^{\delta}\frac{\vert \kappa(t,s,\xi(s),\xi(s-h),\upsilon(s))\vert}{(t-s)^{1-\nu}}ds \bigg)^{p}dt\Bigg]^{1/p}\\
			\leq&\Vert \zeta(\cdot)\Vert_{p}+\Bigg[ \int_{\delta}^{2\delta}\bigg( \int_{0}^{\delta}\frac{L_{0}(s)+L(s)\big[\vert \xi(s)\vert+\vert \xi(s-h)\vert+d(\upsilon(s),\upsilon_{0})\big]}{(t-s)^{1-\nu}}ds \bigg)^{p}dt\Bigg]^{1/p}\\
			\leq&\Vert \zeta(\cdot)\Vert_{p}+\Bigg[ \int_{0}^{2\delta}\bigg( \int_{0}^{\delta}\frac{L_{0}(s)}{(t-s)^{1-\nu}}ds \bigg)^{p}dt\Bigg]^{1/p}\\
			+&\Bigg[ \int_{0}^{2\delta}\bigg( \int_{0}^{\delta}\frac{L(s)\big[\vert \bar{z}(s)\vert+\vert \bar{z}(s-h)\vert+d(\upsilon(s),\upsilon_{0})\big]}{(t-s)^{1-\nu}}ds \bigg)^{p}dt\Bigg]^{1/p}\\
			\leq&\Vert \zeta(\cdot)\Vert_{p}+K\bigg[ 1+2\Vert \xi(\cdot)\Vert_{L^{p}(-h,\delta;R^{n})}
			+2\Vert z(\cdot)\Vert_{L^{p}(\delta,2\delta;R^{n})}+\Vert d(\upsilon(\cdot),\upsilon_{0})\Vert_{p}\bigg]
		\end{align*}
		Following a similar argument to the proof of (4.11), we can conclude that the operator $\mathscr{T}$ maps functions from $L^{p}(\delta,2\delta;R^{n})$ to $L^{p}(\delta,2\delta;R^{n})$.Suppose we have two functions, $z_{1}(\cdot)$ and $z_{2}(\cdot)$, both belonging to $L^{p}(\delta,2\delta;R^{n})$. By using Corollary \ref{c1}, we obtain
		\begin{align*}
			&	\Big\Vert \mathscr{T}[z_{1}(\cdot)](t)-\mathscr{T}[z_{2}(\cdot)](t)\Big\Vert_{L^{p}(\delta,2\delta;R^{n})}\equiv\bigg[ \int_{\delta}^{2\delta} \bigg\vert \mathscr{T}[z_{1}(\cdot)](t)-\mathscr{T}[z_{2}(\cdot)](t)\bigg\vert^{p}dt\bigg]^{1/p}\\
			=&\bigg[ \int_{\delta}^{2\delta} \bigg\vert \int_{\delta}^{t} \frac{\kappa(t,s,z_{1}(s),z_{1}(s-h),\upsilon(s))-\kappa(t,s,z_{2}(s),z_{2}(s-h),\upsilon(s))}{(t-s)^{1-\nu}}ds\bigg\vert^{p}dt\bigg]^{1/p}\\
			\leq&\bigg[\int_{\delta}^{2\delta}\bigg(\int_{\delta}^{t}\frac{L(s)[\vert z_{1}(s)-z_{2}(s)\vert+\vert z_{1}(s-h)-z_{2}(s-h)\vert]}{(t-s)^{1-\nu}}ds\bigg)^{p}dt\bigg]^{1/p}\\
			\leq& \bigg(\frac{\delta^{1-(1+\varepsilon)(1-\nu)}}{1-(1+\varepsilon)(1-\nu)}\bigg)^{\frac{1}{1+\varepsilon}}\Vert L(\cdot)\Vert_{L^{\frac{p q}{p-q}}(0,T)} \Big[\Vert z_{1}(\cdot)-z_{2}(\cdot)\Vert_{L^{p}(\delta,2\delta;R^{n})}\\
			+&\Vert z_{1}(\cdot-h)-z_{2}(\cdot-h)\Vert_{L^{p}(\delta,2\delta;R^{n})}\Big]\\
			\leq& 2\bigg(\frac{\delta^{1-(1+\varepsilon)(1-\nu)}}{1-(1+\varepsilon)(1-\nu)}\bigg)^{\frac{1}{1+\varepsilon}}\Vert L(\cdot)\Vert_{L^{\frac{p q}{p-q}}(0,T)} \Vert z_{1}(\cdot)-z_{2}(\cdot)\Vert_{L^{p}(\delta,2\delta;R^{n})}
		\end{align*}
		Hence we get the existence and uniqueness of the solution to the state equation on $[\delta,2\delta]$. Using the induction, we can obtain the solution $\xi(\cdot)\in [-h,\delta],[\delta,2\delta],\cdots,[(\frac{T}{\delta})\delta,T]$.
		
		Now, suppose $\zeta_{1}(\cdot), \upsilon_{1}(\cdot)), \quad(\zeta_{2}(\cdot), \upsilon_{2}(\cdot))\in L^{p}(-h,T;R^{n})\times \mathscr{U}^{p}[0,T]$ and $\xi_{1}(\cdot),\xi_{2}(\cdot)$ become the corresponding solutions. Hence
		\begin{align*}
			\vert \xi_{1}(t)-\xi_{2}(t)\vert\leq& \vert \zeta_{1}(t)-\zeta_{2}(t)\vert+\int_{0}^{t}\frac{\vert \kappa(t,s,\xi_{1}(s),\xi_{1}(s-h),\upsilon_{1}(s))-\kappa(t,s,\xi_{2}(s),\xi_{2}(s-h),\upsilon_{2}(s))\vert}{(t-s)^{1-\nu}}ds\\
			+&\int_{0}^{t}\frac{\vert \kappa(t,s, \xi_{1}(s), \xi_{1}(s-h),\upsilon_{2}(s))-\kappa(t,s, \xi_{2}(s), \xi_{2}(s-h,\upsilon_{2}(s))\vert}{(t-s)^{1-\nu}}ds\\
			\leq&\vert \zeta_{1}(t)-\zeta_{2}(t)\vert+\int_{0}^{t}\frac{\vert \kappa(t,s,\xi_{1}(s),\xi_{1}(s-h),\upsilon_{1}(s))-\kappa(t,s,\xi_{2}(s),\xi_{2}(s-h),\upsilon_{2}(s))\vert}{(t-s)^{1-\nu}}ds\\
			+&\int_{0}^{t}\frac{L(s)\Big[\vert \xi_{1}(s)-\xi_{2}(s)\vert+\vert \xi_{1}(s-h)-\xi_{2}(s-h)\vert\Big]}{(t-s)^{1-\nu}}ds\\
			\equiv& \vartheta(t)+\int_{0}^{t}\frac{L(s)\vert \xi_{1}(s)-\xi_{2}(s)\vert}{(t-s)^{1-\nu}}ds+\int_{0}^{t}\frac{L(s)\vert \xi_{1}(s-h)-\xi_{2}(s-h)\vert}{(t-s)^{1-\nu}}ds
		\end{align*}
		Then, by using Theorem \ref{t12}, we will get
		\begin{align}
			\vert \xi_{1}(t)-\xi_{2}(t)\vert\leq&\vartheta(t)+K\sum_{k=0}^{n}\int_{0}^{t-kh}\frac{L(s)\vartheta(s)}{(t-kh-s)^{1-\nu}}ds+K\sum_{k=0}^{n-1}\int_{h}^{t-kh}\frac{L(s)\vartheta(s-h)}{(t-kh-s)^{1-\nu}}ds, \quad a.e. \quad t\in [0,T],
		\end{align}
		for some constant $K>0$. As a result,
		\begin{align*}
			\Vert  \xi_{1}(\cdot)-\xi_{2}(\cdot)\Vert_{p}\leq& K	\Bigg[\int_{0}^{T}\Big(\vartheta(t)+K\sum_{k=0}^{n}\int_{0}^{t-kh}\frac{L(s)\vartheta(s)}{(t-kh-s)^{1-\nu}}ds+K\sum_{k=0}^{n-1}\int_{h}^{t-kh}\frac{L(s)\vartheta(s-h)}{(t-kh-s)^{1-\nu}}ds\Big)^{p}dt\Bigg]^{1/p}\\
			\leq&K\Big(\int_{0}^{T}\big(\vartheta(t)\big)^{p}dt\Big)^{1/p}\leq K\Bigg\lbrace \Vert \zeta_{1}(t)-\zeta_{2}(t)\Vert_{p}\\
			+&\bigg[\int_{0}^{T}\bigg(\int_{0}^{t}\frac{\vert \kappa(s, \xi_{1}(s), \xi_{1}(s-h),\upsilon_{1}(s))-\kappa(s, \xi_{2}(s), \xi_{2}(s-h,\upsilon_{2}(s))\vert}{(t-s)^{1-\nu}}ds\bigg)^{p}dt\bigg]^{1/p}\Bigg\rbrace.
		\end{align*}
		proving the stability estimate. We are able to the similar argument to prove (4.7).
		
		\textbf{Case 2.} $1<p\leq \frac{1}{1-\nu}$. In this case,
		\begin{align*}
			\frac{1}{\nu}\leq \frac{p}{p-1}, \quad \frac{p}{1+\nu p}\leq 1.
		\end{align*}
		Also, since $1-\nu\leq \frac{1}{p}<1$ for each $\varepsilon\in (0,p-1)$, the following satisfies:
		
		\begin{align*}
			1-\nu\leq\frac{1}{p}<\frac{1}{1+\varepsilon}.
		\end{align*}
		This imply $(1-\nu)(1+\varepsilon)<1$. Hence
		\begin{align*}
			\frac{1}{p}<\frac{1}{q}=\frac{1}{p}+1-\frac{1}{1+\varepsilon}\nearrow1, \quad \frac{p-q}{p q}=\frac{1}{p}-\frac{1}{q}=1-\frac{1}{1+\varepsilon}\nearrow \frac{p-1}{p}, \quad as \quad \varepsilon\nearrow p-1.
		\end{align*}
		Therefore, we can choose a value for $\varepsilon$ close to $p-1$, given that $L_{0}(\cdot)\in L^{1+}(0,T)$ and $L(\cdot)\in L^{\frac{p}{p-1}+}(0,T)$. This will ensure that $L_{0}(\cdot)\in L^{q}(0,T)$ and $L(\cdot)\in L^{\frac{p q}{p-q}}(0,T)$. The remaining steps of the proof will be identical to Case 1.
		
		\textbf{Case 3.} $p=1$. In this part, the condition reads $L_{0}(\cdot)\in L^{1+}(0,T)$ and $L(\cdot)\in L^{\infty}(0,T)$. Hence we take $\varepsilon=0$, and (4.11) reads
		\begin{align}
			\Vert \mathscr{T}[z(\cdot)]\Vert_{L^{1}(-h,S;R^{n})}\leq \Vert \zeta(\cdot)\Vert_{1}+\frac{S^{\nu}}{\nu}\lbrace \Vert L_{0}(\cdot)\Vert_{L^{1}(0,T)}+\Vert L(\cdot)\Vert_{L^{\infty}(0,T)}(2\Vert  z(\cdot)\Vert+\Vert d(\upsilon(\cdot),\upsilon_{0})\Vert_{L^{1}(0,S)})\rbrace.
		\end{align}
		
		The remainder of the proof is comparable to that of Case 1.
		
	\end{proof}
	
	\begin{example}
		Let
		\begin{align*}
			\kappa(t,s,\xi(s),\xi(s-h),\upsilon)=\frac{\sqrt{\vert s-1\vert^{2\delta-2}\vert t+1\vert^{2-2\gamma}+\vert \xi(s)\vert+\vert \xi(s-h)\vert}}{\vert s-1\vert^{1-\nu}\vert t+1\vert^{1-\gamma}},\\
			\forall (t,s,\xi,\xi_{s-h},\upsilon)\in\Delta\times R^{2}\times U,\quad  \gamma<1, \quad s\not=1.
		\end{align*}
		Consider the following Volterra integral equation
		\begin{equation}\label{14}
			\xi(t)=\frac{1}{\vert t-1\vert^{1-\sigma}}+\int_{0}^{t} \frac{\sqrt{\vert s-1\vert^{2\delta-2}\vert t+1\vert^{2-2\gamma}+\vert \xi(s)\vert+\vert \xi(s-h)\vert}}{\vert s-1\vert^{1-\nu}\vert t+1\vert^{1-\gamma}(t-s)^{1-\beta}} ds, \quad a.e,\quad t\in[0,T]
		\end{equation}
		for some $\nu, \beta\in (0,1)$, $\sigma,\delta\in (0,1]$ , and $\gamma>1$. We can take
		\begin{align*}
			\zeta(t)=\frac{1}{\vert t-1\vert^{1-\sigma}},\quad L_{0}(s)=\frac{1}{\vert s-1\vert^{2-\nu-\delta}},\quad L(s)=\frac{1}{\vert s-1\vert^{1-\nu}}, \quad t\not=1,\quad s\not=1.
		\end{align*}
		We see that
		\begin{align*}
			\zeta(\cdot)\in L^{p}(0,1)\leftarrow p(1-\sigma)<1 \iff p<\frac{1}{1-\sigma} \quad (\frac{1}{0}=\infty),
		\end{align*}
		\begin{align*}
			\begin{cases}
				(2-\nu-\delta)\frac{p}{1+\beta p}<1 \iff 2-\nu-\beta-\delta<\frac{1}{p}\\
				\\
				\quad\quad\quad\quad\quad\quad\quad\quad\quad \iff p<\frac{1}{(2-\nu-\beta-\delta)_{+}},\\
				\\
				2-\nu-\delta<1\quad \quad\quad \iff \nu+\beta>1.
			\end{cases}
		\end{align*}
		\begin{align*}
			\implies L(\cdot)\in L^{(\frac{1}{\beta} \vee 1)+}(0,T),
		\end{align*}
		\begin{align*}
			\begin{cases}
				\frac{1-\nu}{\beta}<1 \iff \nu+\beta>1,\\
				\\
				(1-\nu)\frac{p}{p-1}<1 \iff1-\nu<1-\frac{1}{p}\iff p>\frac{1}{\nu}
			\end{cases}
		\end{align*}
		\begin{align*}
			L(\cdot)\in L^{(\frac{1}{\beta} \vee \frac{p}{p-1})+}(0,T).
		\end{align*}
		Then, (\ref{14}) has a unique solution $\xi(\cdot)\in L^{p}(0,T)$ for any
		\begin{align*}
			p \in \Big( \frac{1}{\nu}, \frac{1}{(1-\sigma)\vee (2-\nu-\beta-\delta)+}\Big),
		\end{align*}
		supplied
		\begin{align}
			\nu+\beta>1, \quad \nu+\delta>1.
		\end{align}
		
		Let us highlight that the solution $\xi(\cdot)$ of equation (4.10) does not have to be continuous in general, even if the free term $\zeta(\cdot)$ is continuous. Specifically, when we set $\sigma$ to 1 and consider $\zeta(t)$ as a constant function equal to 1, which is continuous, we observe that the solution $\xi(\cdot)$ is positive. This positivity can be determined by applying the Picard iteration method. Thus,
		\begin{align}
			\lim_{t\to 1}\xi(t)\geq 1+\lim_{t\to 1} \int_{0}^{t}\frac{ds}{\vert s-1\vert^{2-\nu-\delta}(t-s)^{1-\beta}}=\int_{0}^{1}\frac{ds}{(1-s)^{3-\nu-\beta-\delta}}=\infty,
		\end{align}
		supplied
		\begin{align}
			3-\nu-\beta-\delta>1 \iff \nu+\beta+\delta<2.
		\end{align}
		This will be the case if we take
		\begin{align*}
			\nu=\frac{2}{3}, \quad \beta=\delta=\frac{1}{2}.
		\end{align*}
		In this scenario, there is a solution $\xi(\cdot)\in L^{p}(0,T)$ where $p \in (\frac{3}{2}, 3)$. However, it should be noted that the solution is not continuous at $t=1$.
		
	\end{example}

\end{document}